\documentclass[12pt]{amsart}
\usepackage{amsmath,amssymb,amsxtra,anysize}
\usepackage{xypic}

\theoremstyle{plain}
\newtheorem{theorem}[equation]{Theorem}
\newtheorem{lemma}[equation]{Lemma}
\newtheorem{proposition}[equation]{Proposition}
\newtheorem{corollary}[equation]{Corollary}

\theoremstyle{definition}
\newtheorem{definition}[equation]{Definition}
\newtheorem{example}[equation]{Example}
\newtheorem{remark}[equation]{Remark}

\DeclareMathOperator{\Ker}{Ker}
\DeclareMathOperator{\Imm}{Im}
\newcommand{\C}{\mathbb{C}}
\newcommand{\Z}{\mathbb{Z}}

\newcommand{\CC}{\mathcal{C}}
\newcommand{\wCC}{\widetilde{\mathcal{C}}}
\newcommand{\CO}{\mathcal{O}}
\newcommand{\CM}{\mathcal{M}}
\newcommand{\E}{\mathcal{E}}
\newcommand{\F}{\mathcal{F}}
\newcommand{\T}{\mathcal{T}}
\newcommand{\Om}{\Omega}

\newcommand{\ind}{{\rm ind}\,}
\newcommand{\indGSV}{{\rm ind}_{\rm GSV}\,}

\newcommand{\indhom}{{\rm ind}_{\rm hom}\,}

\newcommand{\Ch}{{\rm Ch}\,}

\title{Homological indices of collections of 1-forms.}

\author{E.~Gorsky} 
\address{Department of Mathematics, University of California at Davis}
\address{National Research University Higher School of Economics, Moscow, Russia}
\email{egorskiy\symbol{'100}math.ucdavis.edu}
\thanks{}
\author{S.M.~Gusein-Zade} 
\address{Moscow State University, Faculty of Mathematics and Mechanics}%%, GSP-1, Moscow, 119991, Russia.} 
\email{sabir\symbol{'100}mccme.ru}
\thanks{The work was supported by the grant 16-11-10018 of the Russian Science Foundation.}

\begin{document}

\maketitle

\begin{abstract}
Homological index of a holomorphic 1-form on a complex analytic variety with an isolated singular point
is an analogue of the usual index of a 1-form on a non-singular manifold. One can say that it corresponds
to the top Chern number of a manifold. We 
%% define homological indices 
offer a definition of homological indices for collections of 1-forms on a
(purely dimensional) complex analytic variety with an isolated singular point corresponding to other
Chern numbers. We also define new invariants of germs of complex analytic
varieties with isolated singular points related to ``vanishing Chern numbers'' at them.
\end{abstract}

%%%%%%%%%%%%%%%%%%%%%%%%%%%%%%%%%%
\section{Introduction.}
%%%%%%%%%%%%%%%%%%%%%%%%%%%%%%%%%%
For an isolated singular point of a vector field or of a 1-form on a smooth manifold one has a well-known
integer invariant~-- the index. It can be defined for vector fields or 1-forms on a complex-analytic
manifold as well. The notions of the indices of isolated singular
points of a vector field or of a 1-form have several generalizations to singular (real or complex) analytic
varieties: see, e.g., \cite{EG-Survey}. In particular, there are defined (in somewhat different settings)
the so called GSV-index, the radial index and the Euler obstruction of a 1-form.
One of the generalizations for a 1-form on a complex analytic variety with an isolated
singular point at the origin is the so called homological index: \cite{EGS}.

The sum of indices of (isolated) singular points of a vector field or of a 1-form on a compact smooth
(differentiable) manifold without boundary is equal to the Euler characterictic of the manifold. The sum
of indices of (isolated) singular points of a complex-valued 1-form on an $n$-dimensional compact complex
analytic manifold $X^n$ is equal to $(-1)^n$ times the Euler characterictic of the manifold which coincides
with the top Chern number $\langle c_n(T^*X^n),[X^n]\rangle$ of the cotangent bundle. Thus one can say
that the indices of singular points of vector fields or of 1-forms on complex analytic varieties correspond
to the top Chern number.

Other Chern numbers correspond to indices of singular points of collections of 1-forms on varieties.
On an $n$-dimensional compact complex analytic manifold $X^n$ the Chern number
$\langle \prod_{i=1}^s c_{k_i}(T^*X^n),[X^n]\rangle$ with $\sum_{i=1}^s k_i=n$ is equal to
the sum of the (properly defined) indices of isolated singular points of a collection
$\{\omega_j^{(i)}\}$ of 1-forms ($i=1,\ldots, s$, $j=1,\ldots, n-k_i+1$): see, e.g., \cite{EG-BLMS-2005}.
Analogues of the GSV-index for collections of 1-forms
on an isolated complete intersection singularity were defined in \cite{EG-BLMS-2005}. If all the
forms in the collection are complex analytic, this index is expressed as the dimension of a certain algebra.
In \cite{EG-collections}, there was defined an analogue of the Euler obstruction for a collection
of 1-forms on a purely $n$-dimensional complex analytic variety called the Chern obstruction.

Here we offer a definition of homological indices for collections of 1-forms on a (purely dimensional)
complex analytic variety with an isolated singular point corresponding to Chern numbers different from
the top one. We also define new invariants of germs of complex analytic varieties with isolated singular
points related to ``vanishing Chern numbers'' at them.

\section*{Acknowledgments}
We are grateful to A.G. Alexandrov for useful discussions. 

%%%%%%%%%%%%%%%%%%%%%%%%%%%%%%%%%%%%%%%%%%%%%%%%%
\section{Homological index of a 1-form.}
%%%%%%%%%%%%%%%%%%%%%%%%%%%%%%%%%%%%%%%%%%%%%%%%%
For a germ of a holomorphic 1-form $\omega=\sum\limits_{i=1}^n A_i(\overline{z})dz_i$ with an isolated
singular point (zero) at the origin
in $\C^n$ its index $\ind(\omega; \C^n,0)$ is equal to the multiplicity of the map
${\mathcal{A}}=(A_1, \ldots, A_n):(\C^n,0)\to(\C^n,0)$ which, in turn, is equal to
$$
\dim_{\C}\left(\CO_{\C^n,0}/\langle A_1,\ldots, A_n\rangle\right)\,,
$$ where $\CO_{\C^n,0}$
is the ring of germs of holomorphic functions on $\C^n$ at the origin, $\langle A_1,\ldots, A_n\rangle$
is the ideal generated by the corresponding elements. This relation can be reformulated in the
following way. Let $\Omega^i_{\C^n,0}$ be the module of germs of (holomorphic) differential $i$-forms
on $(\C^n,0)$. Then
$$
\ind(\omega; \C^n,0)=\dim_{\C}\left(\Omega^n_{\C^n,0}/\omega\wedge\Omega^{n-1}_{\C^n,0}\right)\,.
$$
An analogue of this equation holds for a 1-form on an isolated complete intersection singularity (ICIS)
as well. Let $(X,0)\subset (\C^N, 0)$ be an ($n$-dimensional) isolated complete intersection singularity
defined by the equations $f_1=f_2=\ldots=f_{N-n}=0$, $f_i\in{\mathcal O}_{\C^N,0}$, and let
$\Omega^i_{X,0}=\Omega^i_{\C_N,0}/\langle f_i\Omega^i_{\C_N,0}, df_i\wedge\Omega^{i-1}_{\C_N,0}\rangle$
be the module of germs of differential $i$-forms on $(X,0)$.
For a 1-form (not necessarily holomorphic) with an isolated singular point at the origin its GSV-index
$\indGSV(\omega; X,0)$ was defined in \cite{EG-MMJ}.
If the 1-form $\omega$ is holomorphic, one has \cite{EGS}:
\begin{equation}\label{indGSV}
 \indGSV(\omega; X,0)=\dim_{\C}\left(\Omega^n_{X,0}/\omega\wedge\Omega^{n-1}_{X,0}\right).
\end{equation}
%Equation~(\ref{indGSV}) is badly adapted to be used as a definition of an index of a 1-form
%on an arbitrary germ of a complex analytic variety of pure dimension $n$ with an isolated singular
%point at the origin because of the following reason. 
The usual index and the GSV-index possess
the following ``law of conservation of number''. If $\omega'$ is a small deformation of the 1-form
$\omega$, the sum of indices of the singular points of the 1-form $\omega'$ split from the origin
is equal to the index of the 1-form $\omega$ at the origin. 
%% REFERENCE HERE?
This means that %(for the GSV-index)
\begin{equation}
\label{conservation}
\indGSV(\omega; X,0) = \indGSV(\omega'; X,0)
 + \sum_{x} \indGSV(\omega';X,x),
\end{equation}
where the sum on the right hand side is over all singular points $x$ of the
1-form $\omega'$ in a small punctured neighbourhood of the origin $0$ in $X$.
(Let us recall that for all points $x$ from a punctured neighbourhood of the origin $0$ in $X$
the GSV-index $\indGSV(\omega';X,x)$ is actually the usual index
of  the 1-form $\omega'$ on the complex analytic manifold $X\setminus\{0\}$.)

This property does not hold in general for a 1-form on a germ of a complex analytic variety with
an isolated singular point at the origin if the index is defined by Equation~(\ref{indGSV}).
A way to bypass this problem is to consider the homological index of a 1-form \cite{EGS}.

Let $(X,0)\subset (\C^N, 0)$ be a germ of a complex analytic variety of pure
dimension $n$ with an isolated singular point at the origin and let
$\omega$ be a holomorphic 1-form on $(X,0)$ (that is the restriction to $(X,0)$
of a holomorphic 1-form on $(\C^N, 0)$) without singular points (zeroes) outside of the origin.
Let $\Omega^i_{X,0}$ be the module of germs of differential $i$-forms on $(X,0)$.
%(If $(X,0)$ is defined by the equations $f_1=f_2=\ldots=f_M=0$, $f_i\in{\mathcal O}_{X,0}$,
%then $\Omega^i_{X,0}=\Omega^i_{\C_N,0}/\langle f_i\Omega^i_{\C_N,0}, df_i\wedge\Omega^{i-1}_{\C_N,0}\rangle$,
%$\Omega^0_{X,0}={\mathcal O}_{X,0}$.)
Let us consider the complex $(\Omega^{\bullet}_{X,0}, \wedge\omega)$: 
\begin{equation}\label{hom-complex}
0 \to \Omega^0_{X,0}
\to \Omega^1_{X,0} \to ... \to \Omega^n_{X,0} \to 0\,, 
\end{equation}
where the arrows are the exterior products by the 1-form $\omega$: $\wedge\omega$.
This complex has finite-dimensional (co)homology groups $H^i(\Omega^\bullet_{X,0},\wedge\omega)$.
(This follows from the fact that the corresponding complex of sheaves consists of coherent
sheaves and its cohomologies are concentrated at the origin.)

\begin{definition} (\cite{EGS})
The {\em homological index} of the 1-form $\omega$ on $(X,0)$ is defined by
\begin{equation}\label{eq1}
\indhom(\omega; X,0)
= \sum_{i=0}^n (-1)^{n-i} \dim_{\C} H^i(\Omega^\bullet_{X,0},\wedge\omega)\,.
\end{equation}
\end{definition}

 The homological index satisfies the law of conservation of number \cite{GG}.
 If $X$ is an ICIS, the homological index of a holomorphic 1-form coincides with its GSV-index \cite[Theorem 3.2]{EGS}.

 In \cite{GM} there was considered an equivariant (with respect to a finite group action) version of
 the homological index of a 1-form. It take values in the ring of representations of the group.
 It was shown that on a smooth manifold this index coincides with the reduction of the equivariant index
 with values in the Burnside ring of the group defined earlier.
 
%%%%%%%%%%%%%%%%%%%%%%%%%%%%%%%%%%%%%%%%%%%%%%%%%
\section{Indices of collections of forms.}
%%%%%%%%%%%%%%%%%%%%%%%%%%%%%%%%%%%%%%%%%%%%%%%%%
%% Let $(X,0)\subset (\C^N, 0)$ be a germ of a complex analytic variety of pure
%% dimension $n$ with an isolated singular point at the origin.
Let $k_i$, $i=1,\ldots, s$, be positive integers such that $\sum\limits_{i=1}^s k_i=n$.
We shall consider collections of 1-forms $\{\omega_j^{(i)}\}$, $i=1,\ldots, s$, $j=1,\ldots,n-k_i+1$,
on (purely) $n$-dimensional varieties or on germs of $n$-dimensional varieties. 
One can say that collections of this sort correspond to the Chern number
$\langle \prod_{i=1}^s c_{k_i},[\bullet]\rangle$ in the following sense.
Let $X$ be a (non-singular) compact complex manifold of dimension $n$ and let $\{\omega_j^{(i)}\}$,
$i=1,\ldots, s$, $j=1,\ldots,n-k_i+1$, be a collection of 1-forms on it (continuous, but not necessarily
holomorphic). 
\begin{definition}
A point $x\in X$ is called {\em a singular point} of the collection $\{\omega_j^{(i)}\}$
if for each $i$ the 1-forms $\omega_1^{(i)}$, \dots, $\omega_{n-k_i+1}^{(i)}$ at the point $x$ are
linearly dependent. 
\end{definition}
For an isolated singular point $x$ of a collection $\{\omega_j^{(i)}\}$ one can define
the notion of its index ${\rm ind}(\{\omega^{(i)}_j\}; X,x)$ (see a more general definition for a collection
of 1-forms on an ICIS below). 
If the collection $\{\omega_j^{(i)}\}$ has only isolated singular points on $X$,
the sum of their indices is equal to the characteristic number
$\langle \prod_{i=1}^s c_{k_i}(T^*V^n),[V^n]\rangle$.

For positive integers $N$ and $m$ with $N \geq m$, let ${\CM}(N,m)$ be the space of $N \times m$ matrices
with complex entries and let $D(N,m)$ be the subspace of ${\CM}(N,m)$ consisting of degenerate matrices,
that is of matrices of rank less than $m$.
(The subset $D(N,m)$ is an irreducible subvariety of  ${\CM}(N,m)$ of codimension $N-m+1$.)
For a sequence $\widehat{\bf m}=(m_1, \ldots , m_s)$ of positive integers, 
let ${\CM}_{N, \widehat{\bf m}}:= \prod_{i=1}^s {\CM}(N,N-m_i+1)$ and let $D_{N, \widehat{\bf m}}:=
\prod_{i=1}^s D(N,N-m_i+1)$. 
The variety $D_{N, \widehat{\bf m}}$ is irreducible of codimension $m=\sum_im_i$, and therefore its
complement $W_{N, \widehat{\bf m}}= {\CM}_{N, \widehat{\bf m}} \setminus D_{N, \widehat{\bf m}}$ is
$(2m-2)$-connected, $H_{2m-1}(W_{N, {\bf k}}) \cong \Z$, and there is a natural choice of a
generator of the latter group. This choice defines the degree (an integer) of
a map from an oriented manifold of dimension $2m-1$ to %% the manifold 
$W_{N, \widehat{\bf m}}$.

Let, as above, $(X,0)\subset (\C^N, 0)$ be an ($n$-dimensional) isolated complete intersection singularity
defined by the equations $f_1=f_2=\ldots=f_{N-n}=0$, $f_i\in{\mathcal O}_{\C^N,0}$.
Let $k_i$, $i=1,\ldots, s$, be positive integers such that $\sum\limits_{i=1}^s k_i=n$
and $\{\omega_j^{(i)}\}$, $i=1,\ldots, s$, $j=1,\ldots,n-k_i+1$, be a collection of 1-forms on $(X,0)$
(that is restrictions to $(X,0)$ of 1-forms on $(\C^N, 0)$) without singular points on $X$ outside the origin..
%% (Here it is not necessary to demand that the 1-forms $\omega^{(i)}_j$
%% are complex analytic. It is sufficient to suppose that $\omega^{(i)}_j$ are continuous
%% complex linear functions on the tangent bundle $T\C^n$.) 
Let $U$ be a neighbourhood of the origin in $\C^n$ where all the functions
$f_r$ ($r=1, \ldots , N-n$) and the 1-forms $\omega_j^{(i)}$ are defined and such
that the restriction of the collection $\{\omega_j^{(i)} \}$ of 1-forms to $(X \cap
U) \setminus \{ 0\}$ has no singular points. Let $S_\delta \subset U$ be a
sufficiently small sphere around the origin which intersects $X$ transversally
and denote by $K=X \cap S_\delta$ the link of the ICIS $(X,0)$. 
Let $\widehat{\bf k}:= (k_1, \ldots , k_s)$ and let $\Psi_X$ be the map from $X \cap U$ to
${\CM}_{n,\widehat{\bf k}}$ which sends a point $x \in X \cap U$ to the collection of
$N \times (N-k_i+1)$-matrices 
$$
\{(df_1(x), \ldots , df_{N-n}(x), \omega_1^{(i)}(x), \ldots ,\omega_{n -k_i+1}^{(i)}(x)) \},
\quad i=1, \ldots, s.
$$
Its restriction $\psi_X$ to the link $K$ maps $K$ to $W_{N,\widehat{\bf k}}$. 
The following notion was introduced in \cite{EG-BLMS-2005}.

\begin{definition}
The {\em GSV index} $\indGSV(\{\omega^{(i)}_j\};X,0)$ of the
collection of 1-forms $\{ \omega^{(i)}_j \}$ on the ICIS $(X,0)$ is the degree of the mapping
$\psi_X : K \to W_{N, \widehat{\bf k}}$ or, equivalently, %(what is the same)
the intersection number $(\Imm{\Psi_X}\circ D_{N, \widehat{\bf k}})$.
\end{definition}

Assume now that all the 1-forms $\omega^{(i)}_j$ in the collection are complex analytic.
In this case one has the following (``algebraic'') formula for the index $\indGSV(\{\omega^{(i)}_j\};X,0)$.
Let $I_{X,\{ \omega_j^{(i)}\}}$ be the ideal in the ring
${\CO}_{\C^n,0}$ generated by the functions $f_1, \ldots , f_{N-n}$ and by the
$(N-k_i+1) \times (N-k_i+1)$ minors of all the matrices
$$
(df_1(x), \ldots , df_{N-n}(x), \omega_1^{(i)}(x), \ldots ,\omega_{n-k_i+1}^{(i)}(x))
$$
for all $i=1, \ldots, s$. Then one has (\cite{EG-BLMS-2005})
\begin{equation}
\label{def gsv collection}
{\indGSV}(\{ \omega^{(i)}_j \};X,0) = \dim_\C {\CO}_{\C^n,0}/I_{X,\{\omega_j^{(i)}\}}.
\end{equation}

Let $(X,0)\subset (\C^N,0)$ be an arbitrary germ of an analytic variety (not necessarily with
an isolated singularity) and let $\{\omega_j^{(i)}\}$, $i=1,\ldots, s$, $j=1,\ldots,n-k_i+1$,
be a collection of 1-forms on $(X,0)$ ($\omega_j^{(i)}$ is the restriction of a 1-form on $(\C^N,0)$
which will be denoted by $\omega_j^{(i)}$ as well.)

\begin{definition}
A point $x\in X$ is called a {\em special} point of the collection $\{\omega^{(i)}_j\}$ if there
exists a sequence $\{x_m\}$ of points from the non-singular part $X_{\rm reg}$ of the variety $X$
converging to $x$ such that the sequence $T_{x_m}X_{\rm reg}$ of the tangent spaces at the points $x_m$
has a limit $L$ as $m \to \infty$ (in the Grassmann manifold $G(n,N)$ of $n$-dimensional vector subspaces of $\C^N$)
and the restrictions of the 1-forms $\omega^{(i)}_1$, \dots, $\omega^{(i)}_{n-k_i+1}$ to the subspace
$L\subset T_P\C^N$ are linearly dependent for each $i=1, \ldots, s$.
\end{definition}

The collection $\{\omega^{(i)}_j\}$ of 1-forms has an {\em isolated special point} on the germ $(X,0)$
if it has no special points on $X$ in a punctured neighbourhood of the origin.
(The condition for a special point of a collection of holomorphic 1-forms to be non-isolated is a condition
of infinite codimension.) For a collection of 1-forms on $(X,0)$ with an isolated special point at the
origin there is defined the local Chern obstruction $\Ch(\{\omega_j^{(i)}\};X,0)$: \cite{EG-collections}.
It is defined in terms of the Nash transform $\widehat{X}$ of the variety $X$:
the closure in $\C^N\times G(n,N)$ of the set $\{(x, T_xX_{\rm reg})\}$ for all point $x$ from
the non-singular part $X_{\rm reg}$ of $X$.
%% (see, e.g., \cite{EG-collections}). 
Over the Nash transform $\widehat{X}$ one has the Nash bundle which extends
the tangent bundle over the non-singular part of $X$. The forms $\omega_j^{(i)}$ define sections
of the dual bundle. The Chern obstruction $\Ch(\{\omega_j^{(i)}\};X,0)$ is the primary (and in fact the only)
obstruction to extend these sections from the preimage of the intersection of a small sphere around
the origin with $X$ to the Nash transform $\widehat{X}$ so that 
there are no points where
for each $i=1,\ldots, s$ the
extensions of the forms $\omega_j^{(i)}$, $j=1,\ldots, n-k_i+1$, are %% nowhere
linear dependent.
For a generic collection $\{\omega_j^{(i)}\}$ (in particular, for a collection consisting of 
the differentials of a generic collection of linear functions on $\C^N$) the Chern obstruction is
equal to zero. If $(X,0)$ is non-singular, the Chern obstruction coincides with the (usual) index
of the collection of 1-forms.

%\section{Homological algebra}

\section{Homological index for  a collection of forms}

Let $(X,0)$ be a germ of an algebraic variety of dimension $n$ with an isolated singular point at the origin.
As above, let $k_i$, $i=1,\ldots, s$, be positive integers such that $\sum\limits_{i=1}^s k_i=n$ and
let $\{\omega_j^{(i)}\}$, $i=1,\ldots, s$, $j=1,\ldots,n-k_i+1$, be a collection of 1-forms on $(X,0)$.

Let $W_i=\C^{n-k_i+1}$ be an auxiliary vector space with a basis $u_1,\ldots,u_{n-k_i+1}$.  
We define a chain complex $\CC^{(i)}=\CC(\omega_1^{(i)},\ldots,\omega_{n-k_i+1}^{(i)})$ of sheaves of 
%% $\CO_{X,0}$-modules as following:
$\CO_{X,0}$-modules as following:
\begin{equation}
\label{main c}
\CC^{(i)}_0=\Omega^n_{X,0},\ \CC^{(i)}_t=\Omega^{k_i-t}_{X,0}\otimes S^{t-1}W_i,\ 1\le t\le k_i.
\end{equation}
The differential $d_t:\CC^{(i)}_{t}\to \CC^{(i)}_{t-1}$ is defined by the equations:
\begin{eqnarray*}
\label{main d}
d_1(\beta)&=&\beta\wedge \omega_1^{(i)}\wedge\ldots\wedge \omega_{n-k_i+1}^{(i)},\\
d_t(\beta\otimes \varphi(u))&=&\sum_{l=1}^{n-k_i+1}(\beta\wedge \omega_l^{(i)})\otimes \frac{\partial \varphi}{\partial u_l},\ 2\le t\le k_i.
\end{eqnarray*}
\begin{lemma}
One has $d^2=0$, so $(\CC^{(i)},d)$ is a chain complex.
\end{lemma}
\begin{proof}
One has 
$$
d_0d_1(\beta\otimes u_l)=(\beta\wedge \omega_l^{(i)})\wedge \omega_1^{(i)}\wedge\ldots\wedge \omega_{n-k_i+1}^{(i)}=0, 
$$
and
$$
d_td_{t+1}(\beta\otimes \varphi(u))=
\sum_{l,l'}(\beta\wedge \omega_l^{(i)}\wedge \omega_{l'}^{(i)})\otimes
\frac{\partial^2 \varphi}{\partial u_l\partial u_{l'}}=0,\ t>0.
$$
\end{proof}

We can also define $\CC^{(i)}$ (at least its part of positive degree) using the notion of the {\em exterior power} of a chain complex. Recall that if $\E$ and $\F$ are two chain complexes of modules over a commutative ring, then $\E\otimes_{R}\F\simeq \F\otimes_{R}\E$, and the isomorphism is given by $a\otimes b\mapsto (-1)^{\deg(a)\cdot \deg(b)}b\otimes a$. Using this isomorphism, one can define the action of the symmetric group $S_k$ on $\E^{\otimes k}$, and define $\wedge^k(\E)$ as the sign component for this action. One can check  that the exterior powers of a two-term complex have the form:
\begin{equation}
\label{exterior two term}
\wedge^k\left[A\leftarrow B\right]\simeq \left[\wedge^kA\leftarrow \wedge^{k-1}A\otimes B\leftarrow \cdots\leftarrow A\otimes S^{k-1} B\leftarrow S^kB\right].
\end{equation}
Here we assume that $A$ has homological degree 0 and $B$ has homological degree 1.
\begin{proposition}
One has $\CC^{(i)}_{>0}\simeq \wedge^{k_i-1}\left[\Om^1_{X,0}\leftarrow \CO_{X,0}\otimes W_i\right]$, where the differential in the two-term complex 
$\Om^1_{X,0}\leftarrow \CO_{X,0}\otimes W_i$ over $\CO_{X,0}$ sends $u_i$ to $\omega_i$.
\end{proposition}

\begin{proof}
Follows from \eqref{exterior two term} and the fact that $\wedge^i(\Om^1_{X,0})=\Om^i_{X,0}$ (over  $\CO_{X,0}$). 
\end{proof}

\begin{lemma}
\label{lem: support}
The cohomologies of $\CC^{(i)}$ are supported on the subvariety
$Z(X;\omega_1^{(i)},\ldots,\omega_{n-k_i+1}^{(i)})$ consisting of the point $x\in X$ 
where the forms $\{\omega_j^{(i)}\}$ are linearly dependent.
\end{lemma}

\begin{proof}
Indeed, suppose that at some point $x\in X$ the forms $\omega_j^{(i)}$ are linearly independent,
in particular, neither of them vanishes. Since $X$ has an isolated singularity at the origin,
we can assume that $x$ is a smooth point of $X$. Then we can choose local coordinates at $x$ such that
at this point $\omega_j^{(i)}=dx_j$, and one can easily check that $\CC(dx_1,\ldots,dx_{n-k_i+1})$ is acyclic.
\end{proof}

Now we can define a complex 
$$
\CC=\CC(\{\omega_j^{(i)}\})=\bigotimes_{i=1}^{s}\CC^{(i)},
$$
where the tensor product is taken over $\CO_{X,0}$. Note that by construction the complex $\CC^{(i)}$
has length $k_i$, so the complex $\CC$ has length $\sum_{i=1}^{s}k_i=n$.

\begin{definition}
The homological index of the collection of 1-forms $\{\omega_j^{(i)}\}$ is defined as the
Euler characteristic of the complex $\CC$:
\begin{equation}
\label{def homological}
\indhom\left(\{\omega_j^{(i)}\}\right)=\sum_{t=0}^{n}(-1)^{t}\dim H^{t}(\CC).
\end{equation}
\end{definition}

%% By Lemma \ref{lem: support} the cohomology of $\CC$ are supported at the intersection of
%% $Z(X;\omega_j^{(i)})$, that is, at the set of singular points of $\{\omega_j^{(i)}\}$.
%% If $\{\omega_j^{(i)}\}$ has only isolated singular points, all cohomologies are finite-dimensional.
%% Therefore, the homological index is well-defined.
By Lemma \ref{lem: support}, if the collection $\{\omega_j^{(i)}\}$ has an isolated singular point on $(X,0)$,
at each point of $X$ outside of the origin at least one of the complexes $\CC^{(i)}$
is acyclic and therefore the complex $\CC$ is acyclic as well. This means that the cohomologies
of $\CC$ are supported at the origin and therefore are finite-dimensional.
This implies that the homological index is well-defined.

\begin{example}
Suppose that $k_1=n$, that is the collection consists of a single 1-form $\omega=\omega_1^{(1)}$.
Then the complex $\CC=\CC^{(1)}$  agrees with \eqref{hom-complex}, and the definitions of the homological
index agree.
\end{example}

\begin{proposition}\label{prop:law_of_cons}
The homological index for a collection of 1-forms with an isolated singular point satisfies
the law of conservation of number $($like \eqref{conservation}$)$.
\end{proposition}

\begin{proof}
%% Essentially follows from the main theorem of \cite{GG}.
This is a direct consequence of \cite{GG}.
\end{proof}

Since any collection of holomorphic 1-forms on $(\C^n,0)$ can be deformed to a one with non-degenerate
singular points (that is to a collection with singular points of index 1) and for a non-degenerate singular
point the homological index is equal $1$ as well, Proposition~\ref{prop:law_of_cons} implies that
on a non-singular manifold the homological index coincides with the usual one.
In the next section we shall show that on an isolated complete intersection singularity the 
homological index coincides with the GSV one.

%% \begin{theorem}
%% For complete intersections, the homological index and the GSV-index agree.
%% \end{theorem}
%% 
%% The proof uses Corollary \ref{cor: smooth} which is formulated and proved in Section \ref{sec: vanishing}.
%% 
%% \begin{proof}
%% We can deform all 1-forms in our collection so that all singular points $p_l$ of $\{\widetilde{\omega_j}^{(i)}\}$ are away from the origin.
%% Since both GSV-index and the homological index are additive,  we get
%% $$
%% \ind_{\hom}(X,0,\omega_j^{(i)})=\sum_{l}\ind_{\hom}(X,p_l,\widetilde{\omega}_j^{(i)})=\sum_{l}\ind_{GSV}(X,p_l,\widetilde{\omega}_j^{(i)})=
%% \ind_{GSV}(X,0,\omega_j^{(i)}).
%% $$
%% Indeed, the points $p_l$ are smooth in $X$, and by Corollary \ref{cor: smooth} below  the homological index and the GSV index agree at smooth points.  
%% \end{proof}

\section{The case of complete intersections}

%In this section we prove that the homological index for the collections of 1-forms on complete intersections agrees with the 
%GSV-index defined in \cite{EG-BLMS-2005}. 

\begin{theorem}
Let $(X,0)$ be an isolated complete intersection singularity and let $\{\omega_j^{(i)}\}$ be a collection
of holomorphic 1-forms on $(X,0)$ with an isolated 
singular point. Then the homological index defined by \eqref{def homological} agrees with the GSV-index
defined by \eqref{def gsv collection}.
\end{theorem}

The key role in the proof is played by the classical Eagon-Northcott complex \cite{EN,Eis},
and in the next subsection we remind its definition and properties.

%We denote the zeroth cohomology of $\CC$ by
%$$
%Z(X;\omega_1,\ldots,\omega_{n-k+1}):=
%\frac{\Omega^n_{X,0}}{\omega_1\wedge\cdots\wedge \omega_{n-k+1}\wedge \Omega^{k-1}_{X,0}}.
%$$
%In the next sections we show that under certain assumptions on $(X,0)$ and $\omega_i$
%all other homologies of the complex $\CC$ vanish.

\subsection{Eagon-Northcott complex}

The complex \eqref{main c} is very similar to the so-called Eagon-Northcott complex \cite{EN,Eis}
which we now review.
Let $R$ be a commutative ring, and let $M=(m_{ij})$ be an $s\times r$ matrix ($s\le r$) with entries in $R$.
For $1\le i_1<\ldots<i_s\le r$ we denote by $\Delta_{i_1,\ldots,i_s}(M)$ the corresponding $s\times s$ minor
of $M$.
As above, let $W$ be a $s$-dimensional space with basis $u_1,\ldots,u_s$, and let $V$ be an $r$-dimensional
space with basis $e_1,\ldots, e_r$. The complex $\E(M)$ has the chain groups
\begin{equation}
\label{EN groups}
\E_0=R,\ \E_j=R\otimes \wedge^{s+j-1}V\otimes S^{j-1}W,\ 1\le j\le r-s+1,
\end{equation}
and the differentials $d_j:\E_{j}\to \E_{j-1}$ are given by the equations
\begin{eqnarray*}
\label{EN differential}
d_1(e_{i_1}\wedge \ldots \wedge e_{i_s})&=&\Delta_{i_1,\ldots,i_s}(M),\\
d_j(e_{i_1}\wedge \ldots \wedge e_{i_{s+j-1}}\otimes \varphi(u))&=&
\sum_{l=1}^{s+j-1}\sum_{t=1}^{s}
(-1)^{l-1}m_{i_l,s}e_{i_1}\wedge \ldots \widehat{e_{i_l}}\ldots \wedge e_{i_{s+j-1}}\otimes
\frac{\partial \varphi}{\partial u_{t}}
\end{eqnarray*}
for %% $j\ge 1$.
$j > 1$. One can check that $d^2=0$.

\begin{theorem}$($\cite{EN}$)$
\label{th: EN}
Suppose that $R$ is Noetherian and the depth of the ideal $(\Delta_{i_1,\ldots,i_s}(M))$ equals $r-s+1$. 
Then $H^j(\E,d)=0$ for $j>0$ and 
$$
H^0(\E,d)=R/\langle\Delta_{i_1,\ldots,i_s}(M)\rangle.
$$
\end{theorem}

\begin{corollary}
If $R$ and $R/\langle\Delta_{i_1,\ldots,i_s}(M)\rangle$ are Cohen-Macaulay of dimensions $N$ and $N-r+s-1$
respectively, then $H^j(\E,d)=0$ for $j>0$.
\end{corollary}

%In other words, under the assumptions of this theorem, the complex $(\E,d)$ is a free resolution of the $R$-module
%$R/(\Delta_{i_1,\ldots,i_s}(M)).$

Theorem \ref{th: EN} can be used to study the complex \eqref{main c} on the affine space.
Suppose that $X=\C^n$, then the coefficients of the forms $\omega_i=\sum_{t=1}^{n} m_{it}dx_t$
define an $(n-k+1)\times n$ matrix $M=(m_{it})$. Let $R=\CO_{\C^n,0}$.

\begin{proposition}
The complexes $\CC(\omega_1,\ldots,\omega_k)$ and $\E(M)$ are isomorphic.
\end{proposition}

\begin{proof}
Let us identify the chain groups first:
$$
\CC_j=\Omega^{k-j}_{\C^n}\otimes S^{j-1}W=R\otimes \wedge^{k-j}V^* \otimes S^{j-1}W\simeq 
R\otimes \wedge^{n-k+j}V \otimes S^{j-1}W=\E_j.
$$
Under this identification, the matrix coefficients of $d_1$ are given by the coefficients of the
$(n-k+1)$-form $\omega_1\wedge\ldots \wedge \omega_{n-k+1}$ which are just the $(n-k+1)\times (n-k+1)$
minors of the matrix $M$.
Furthermore, 
$$
d_j(dx_{i_1}\wedge\ldots \wedge dx_{i_{k-j}}\otimes \varphi(u))=
\sum_{t}dx_{i_1}\wedge\ldots \wedge dx_{i_{k-j}}\wedge \omega_t\otimes \frac{\partial \varphi}{\partial u_t}=
$$
$$
= \sum_{t,l}m_{l,t}dx_{i_1}\wedge\ldots \wedge dx_{i_{k-j}}\wedge dx_{l}\otimes
\frac{\partial \varphi}{\partial u_t}.
$$
\end{proof}

\begin{corollary}
Suppose that $Z(\C^n;\omega_1,\ldots,\omega_{n-k+1})$ is Cohen-Macaulay of dimension $k$.
Then $\CC(\omega_1,\ldots,\omega_{n-k+1})$ is a free resolution of the local ring of $Z(\C^n;\omega_1,\ldots,\omega_{n-k+1})$. 
\end{corollary}

\subsection{Forms on complete intersections}

Let $(X,0)=\{f_1=\ldots=f_{N-n}=0\}$ be an isolated complete intersection singularity in $(\C^N,0)$,
and let $\omega_1,\ldots,\omega_{n-k+1}$ be a sequence of 1-forms on $X$.
Let $R=\CO_{X,0}=\CO_{\C^N,0}/(f_1,\ldots,f_{N-n})$, and let $F\simeq \C^{N-n}$.
The following statement is well known, but we present its proof for the reader's convenience.

\begin{lemma}
\label{lem: resolution complete intersection}
The module $\Om^j_{X,0}$ of $j$-forms on $X$, $j\le n$), has the following free resolution over $R$:
\begin{equation}
\label{resolution complete intersection}
\Om^j_{X,0}\simeq 
\left[R\otimes \wedge^{j}\C^N \leftarrow R\otimes \wedge^{j-1}\C^N \otimes F\leftarrow \cdots
\leftarrow R\otimes S^{j}F\right],
\end{equation}
where the differentials are induced by the map $(df_1,\ldots,df_{N-n}):R\otimes F\to R\otimes \C^{N}$
%% $(df_1,\ldots,df_{N-n}): F\to \C^{N}$
which sends the $i$th basis element of $F$ to $df_i$.
\end{lemma}

\begin{proof}
The module of $j$-forms is defined as the $j$-th exterior power  of the module of 1-forms.
The latter has a natural two-term free resolution over $R$:
\begin{equation}
\label{resolution 1 forms}
\Om^1_{X,0}\simeq \left[R\otimes \C^N\xleftarrow{d} R\otimes F\right],
\end{equation}
where the map $d=(df_1,\ldots,df_{N-n}):R\otimes F\to R\otimes \C^N$
sends the $i$th basis element of $F$ to $df_i$.
%% where $d=(df_1,\ldots,df_{N-n}):F\to \C^{N}$.
The maximal minors of $d$ vanish on the set of singular points of $(X,0)$,
which has codimension $n$ in $R$. Therefore by \cite[Theorem 1]{Weyman} for $j\le n$ the free resolution of the module
$\Om^j_{X,0}\simeq \wedge^j(\Om^1_{X,0})$ coincides with the $j$-th exterior power of the complex \eqref{resolution 1 forms},
which by \eqref{exterior two term}  agrees with \eqref{resolution complete intersection}.

\end{proof}

Consider the Eagon-Northcott complex for the collection of 1-forms $df_i$ over $R$:
\begin{equation}
\label{EN for df}
\E(f):=\left[R\otimes \wedge^N\C^N\xleftarrow{ \wedge df_1\wedge\ldots \wedge df_{N-n}} R\otimes \wedge^{n}\C^N \leftarrow R\otimes \wedge^{n-1}\C^N \otimes F\leftarrow \cdots \leftarrow R\otimes S^{n}F\right].
\end{equation}
The top minors of the $(N-n)\times N$ matrix consisting of the coefficients of $df_i$ define the singular set of $X$, which has  dimension $0$ and codimension $n$ in $X$. If $q$ is the maximal degree of nontrivial cohomology of  $\E(f)$, then by \cite[Theorem 1]{EN} one has
$$
n+q=N-(N-n)+1\ \Rightarrow q=1.
$$
Therefore $H^k(\E(f))=0$ for $k\ge 2$. %, and the complexes \eqref{resolution complete intersection}
%do not have higher cohomologies.
We define a map
$$
\widehat{\phi}:\Om^{n}_{\C^N,0}\to R\otimes \wedge^N\C^N,\ \widehat{\phi}(\omega)=\pi(\omega\wedge df_1\wedge\ldots \wedge df_{N-n}).
$$
where $\pi:\CO_{\C^N,0}\to R$ is the natural projection. Since $\widehat{\phi}$ annihilates both the forms divisible by $f_i$ and the forms divisible by $df_i$, it induces a map $\phi:\Om^{n}_{X,0}\to R\otimes \wedge^N\C^N$.

\begin{definition}
We define a pair of modules $\T, \T'$ by the equations:
\begin{equation}
\label{def T}
\T=(R\otimes \wedge^N\C^N)/\Imm(\phi),\quad \T'=\Ker(\phi)
\end{equation}
\end{definition}
 
The modules $\T$ and $\T'$ are supported at the origin and hence are finite-dimensional. It is clear that $\T=H^0(\E(f)), \T'=H^1(\E(f))$, where the complex $\E(f)$ is defined by \eqref{EN for df}. The following lemma follows from the results of \cite[Proposition 1.11]{Greuel}, but we present its proof for completeness. 
 
%% T for Tjurina!
 
\begin{lemma}
\label{lem tau}
For any finite complex $K=(K_i)$ of free $R$-modules one has $\chi(K\otimes_R \T)=\chi(K\otimes_R \T')$.
\end{lemma}

\begin{proof}
Let us compute the Euler characteristic of the complex $K\otimes_R \E(f)$ in two ways. Since the homologies of $\E(f)$ are supported at the origin, the same is true for $K\otimes_R \E(f)$, so the Euler characteristic is finite. 

 First, since all the chain groups are free over $R$, we can apply the results of \cite{GG} to its deformations. If we deform $f_i$ to $\widetilde{f}_i$ so that $\{\widetilde{f_i}=0\}$ becomes smooth, and $\E(\widetilde{f})$ is clearly acyclic. Therefore $K\otimes_R \E(\widetilde{f})$ is acyclic as well, and $\chi(K\otimes_R \E(f))=K\otimes_R \E(\widetilde{f})=0$.
 
 On the other hand, there is a spectral sequence starting at $K\otimes H^*(\E(f))$ and converging to $H^*(K\otimes \E(f))$, so
 $$
 0=\chi(K\otimes \E(f))=\chi(K\otimes H^*(\E(f)))=\chi(K\otimes \T)-\chi(K\otimes \T'). 
 $$
\end{proof}

\begin{corollary}
One has $\dim \T=\dim \T'$.
\end{corollary}

\begin{proof}
Apply Lemma \ref{lem tau} to $K=R$.
\end{proof}

\begin{proposition}
Let $(X,0)$ be an isolated complete intersection singularity and let  $\{\omega_j^{(i)}\}$ be a collection of 
1-forms on $(X,0)$. Then the GSV-index of this collection equals
$$
\indGSV(\{\omega_j^{(i)}\};X,0)=
\dim\frac{\Omega^n_{X,0}}{\sum_{i}\left((\wedge_j \omega_j^{(i)})\wedge \Omega^{k_i-1}_{X,0}\right)}.
$$ 
\end{proposition}

\begin{remark}
For $k_1=n$ this follows from the proof of \cite[Lemma 5.3]{Greuel}. 
\end{remark}

\begin{proof}
By definition, 
$$
\indGSV(\{\omega_j^{(i)}\};X,0)=
\dim\frac{R\otimes \wedge^N\C^N}
{\phi\left(\sum_{i}\left((\wedge_j \omega_j^{(i)})\wedge \Omega^{k_i-1}_{X,0}\right)\right)}=
$$
$$
\dim(\T)+ \dim\frac{\Imm(\phi)}
{\phi\left(\sum_{i}\left((\wedge_j \omega_j^{(i)})\wedge \Omega^{k_i-1}_{X,0}\right)\right)}=
$$
$$
\dim(\T)-\dim(\T')+
\dim\frac{\Omega^n_{X,0}}{\sum_{i}\left((\wedge_j \omega_j^{(i)})\wedge \Omega^{k_i-1}_{X,0}\right)}.
$$
Now the statement follows from Lemma \ref{lem tau}.
\end{proof}

\subsection{Homology vanishing}
\label{sec: vanishing}

\begin{theorem}
\label{vanishing one set}
Suppose that, as above, $(X,0)$ is an isolated complete intersection singularity
and $Z(X;\omega_1,\ldots,\omega_{n-k+1})$
is Cohen-Macaulay of dimension $k$. Then the complex $\CC(\omega_1,\ldots,\omega_{n-k+1})$ defined by
\eqref{main c}  has no higher cohomologies.
\end{theorem}

\begin{proof}
Let $W\cong \C^{n-k+1}$.
For $j\le n$, let us replace $\Om^j_{X,0}$ appearing in the definition of $\CC(\omega_1,\ldots,\omega_{n-k+1})$ by its free
$R$-resolution \eqref{resolution complete intersection}. We obtain the following bicomplex:

\begin{equation}
\label{bicomplex}
\xymatrix{
R\otimes \wedge^N\C^N  &  & & \\
 \Om^{n}_{X,0} \ar@{-->}[u]^{df_1\cdots df_{N-n}}  & \ar@{-->}[ul]_{\ \ \ \ \ df_1\cdots df_{N-n}\wedge \omega_1\cdots\omega_{n-k+1}}  \ar[l]_{\omega_1\cdots\omega_{n-k+1}\ \ \ \ } R\otimes \wedge^{k-1}\C^N & \ar[l]_{d_W} R\otimes \wedge^{k-2}\C^N\otimes W & \ar[l]_{\ \ \ \ \\ \ \ \ \ \ d_W} \ldots\\
  & \ar[u]_{d_F} R\otimes \wedge^{k-2}\C^N\otimes F & \ar[l]_{d_W} \ar[u]_{d_F} R\otimes \wedge^{k-3}\C^N\otimes F\otimes W &  \ar[l]_{\ \ \ \ \ \ \ \ \ \ \ d_W} \ldots  \\
  & \ar[u]_{d_F} R\otimes \wedge^{k-3}\C^N\otimes S^{2}F &\ar[l]_{d_W} \ar[u]_{d_F} R\otimes \wedge^{k-4}\C^N\otimes S^{2}F\otimes W & \ar[l]_{\ \ \ \ \ \ \ \ \ \ \ \ \ \ \ \ \ d_W} \ldots\\
  & \ar[u]_{d_F}\ldots & \ar[u]_{d_F}\ldots & \\
}
\end{equation}
The horizontal differentials $d_W$ are induced by the map 
$$
(\omega_1,\ldots,\omega_{n-k+1}):R\otimes W\to R\otimes\C^N
$$
(which sends the $i$th basis element of $W$ to $\omega_i$),
%% $(\omega_1,\ldots,\omega_{n-k+1}):W\to \C^N$,
and the vertical differentials $d_F$ are induced by the map
$$
(df_1,\ldots,df_{N-n}):R\otimes F\to R\otimes \C^{N}
$$
(which sends the $i$th basis element of $F$ to $df_i$).
%% $(df_1,\ldots,df_{N-n}):F\to \C^{N}$.
 
If one replaces the differential $d_1$ in the complex $\CC$ by its composition with $\phi$:%=\wedge df_1\cdots df_{N-n}$:
$$
\wedge df_1\cdots df_{N-n}\wedge \omega_1\cdots\omega_{n-k+1}:R\otimes \wedge^{k-1}\C^N \to R\otimes
\wedge^n\C^n, 
$$
one obtains just the EN complex associated to the matrix 
$$
M=(df_1,\ldots, df_{N-n},\omega_1,\ldots,\omega_{n-k+1}):R\otimes(F\oplus W)\to R\otimes\C^N.
%% M=(df_1,\ldots, df_{N-n},\omega_1,\ldots,\omega_{n-k+1}):F\oplus W\to \C^N.
$$ 
Indeed, $S^{j}(F\oplus W)=\oplus_{i=0}^{j}S^{i}F\otimes S^{j-i}W$, and the differentials agree. Let us denote the total complex of the resulting bicomplex by $\wCC(\omega_1\cdots\omega_{n-k+1})$.

 By assumption, $\omega_1,\ldots,\omega_{n-k+1}$ are linearly dependent on a codimension $k$ subvariety
 $Z(X;\omega_1,\ldots,\omega_{n-k+1})$  in $X$.  This subvariety is cut out by $(N-k+1)\times (N-k+1)$-minors
 of the $(N-k+1)\times N$ matrix
$M$, so by Theorem \ref{th: EN} the  complex $\wCC(\omega_1\cdots\omega_{n-k+1})$ does not have higher homologies. Now 
$$
\Ker(d_1)=\Ker(\omega_1\wedge \cdots\wedge \omega_k)\subset 
\Ker( df_1\wedge\cdots\wedge df_{N-n}\wedge \omega_1\wedge \cdots\wedge \omega_k)=
\Ker(\phi\circ d_1)\subset \Imm(d_2),
$$
so the original bicomplex has no higher cohomologies as well. 

On the other hand, consider the associated spectral sequence. By Lemma
\ref{lem: resolution complete intersection} the homologies in columns agree with $\Om^k_{X,0}$
and are concentrated in the top row. Therefore the spectral sequence collapses and $\CC(\omega_1\cdots\omega_{n-k+1})$
has no higher cohomologies.
\end{proof}

Consider now a collection of forms $\{\omega_j^{(i)}\}$ on a complete intersection $(X,0)$.
Let $\wCC^{(i)}=\wCC(\omega_1^{(i)},\ldots,\omega_{n-k_i+1}^{(i)})$
where the complex $\wCC$ was defined in the proof of Theorem \ref{vanishing one set}.
Let $$\wCC(\{\omega_j^{(i)}\})=\wCC^{(1)}\otimes_{R}\wCC^{(2)}\cdots\otimes_{R}\wCC^{(s)}.$$

\begin{theorem}
\label{th:vanishing}
Suppose that the collection of forms $\{\omega_j^{(i)}\}$ on a complete intersection $(X,0)$ has
an isolated singular point at the origin. 
Then the homologies $H^t(\wCC(\{\omega_j^{(i)}\})$ vanish for 
%% $t\ge 0$.
$t > 0$.
\end{theorem}

\begin{proof}
Suppose that $(X,0)$ is defined in $(\C^N,0)$ by the equations $f_1=\ldots=f_{N-n}=0$.
As above, consider the locus $Z(X;\omega_1^{(i)}, \ldots, \omega_{n-k_i+1}^{(i)})$ where the forms
$\omega_1^{(i)},\ldots,\omega_{n-k_i+1}^{(i)}$ are linearly dependent on $X$ or, equivalently,
the forms $\omega_1^{(i)},\ldots,\omega_{n-k_i+1}^{(i)},df_1,\ldots,df_{N-n}$ are linearly dependent
in $(\C^N,0)$.
By construction, $Z(X;\omega_1^{(i)}, \ldots, \omega_{n-k_i+1}^{(i)})$ is the intersection of $(X,0)$
with the  determinantal variety defined by the vanishing of maximal minors of an $N\times (N-k_i+1)$ matrix. 
Therefore the codimension of $Z(X;\omega_1^{(i)}, \ldots, \omega_{n-k_i+1}^{(i)})$ in $(X,0)$ is 
%% greater ??? 
less than or equal to $k_i$. On the other hand, by assumption the intersection of
the loci $Z(X;\omega_1^{(i)}, \ldots, \omega_{n-k_i+1}^{(i)})$ for all $i$ is zero-dimensional,
so it has codimension $n=\sum_{i=1}^{s}k_i$. Therefore for all $i$ 
the subvarieties $Z(X;\omega_1^{(i)}, \ldots, \omega_{n-k_i+1}^{(i)})$ have pure dimensions $n-k_i$,
which are equal to their expected dimensions. Therefore all
$Z(X;\omega_1^{(i)}, \ldots, \omega_{n-k_i+1}^{(i)})$ are Cohen-Macaulay \cite{Eis}.
Furthermore, the intersections of $Z(X;\omega_1^{(i)}, \ldots, \omega_{n-k_i+1}^{(i)})$ for various $i$
are all Cohen-Macaulay of correct dimension.

Let us prove by induction in $i$ that the product
$\wCC^{(1)}\otimes_{\CO_{X,0}}\wCC^{(2)}\otimes\cdots\otimes_{\CO_{X,0}}\wCC^{(i)}$ has no higher homologies.
For $i=1$ this follows from the proof of Theorem \ref{vanishing one set}. Assume that this is true for some $1\le i\le s$.
Then $\wCC^{(1)}\otimes_{\CO_{X,0}}\wCC^{(2)}\otimes\cdots\otimes_{\CO_{X,0}}\wCC^{(i)}$ is a
(not free in general) resolution of the structure sheaf of the Cohen-Macaulay scheme
$$
Z=Z(X;\omega_1^{(1)},\ldots,\omega_{n-k_1+1}^{(1)})\cap\ldots\cap Z(X;\omega_1^{(i)},\ldots,\omega_{n-k_i+1}^{(i)}).
$$
Now
\begin{equation}
\label{eq: tensor}
\wCC^{(1)}\otimes_{\CO_{X,0}}\wCC^{(2)}\otimes\cdots\otimes_{\CO_{X,0}}\wCC^{(i)} \otimes_{\CO_{X,0}}\wCC^{(i+1)}
\simeq \CO_{Z,0}\otimes_{\CO_{X,0}}\wCC^{(i+1)}.
\end{equation}
Similarly to the proof of Theorem \ref{vanishing one set}, we can replace \eqref{eq: tensor} by
the complex \eqref{bicomplex} with $R$ replaced by $\CO_{Z,0}$. Since $Z$ and 
$Z\cap Z(X;\omega_1^{(i+1)}, \ldots, \omega_{n-k_{i+1}+1}^{(i+1)})$ are both Cohen-Macaulay and
$Z\cap Z(X;\omega_1^{(i+1)}, \ldots, \omega_{n-k_{i+1}+1}^{(i+1)})$ has codimension $k_{i+1}$ in $Z$,
by Theorem \ref{th: EN} the bicomplex has no higher homologies. Therefore the complex \eqref{eq: tensor}
has no higher homologies as well. 
\end{proof}

\begin{corollary}
\label{cor: smooth}
%% Suppose that $(X,0)$ is smooth, and the forms $\{\omega_j^{(i)}\}$ have an isolated singular point at the origin. 
%% Then the homological and the GSV indices for this collection of 1-forms agree:
%% $$
%% \ind_{\hom}(X,0,\omega_j^{(i)})=\ind_{GSV}(X,0,\omega_j^{(i)}).
%% $$
Suppose that $(X,0)$ is an isolated complete intersection singularity, and the collection
$\{\omega_j^{(i)}\}$ of holomorphic 1-forms has an isolated singular point at the origin. 
Then the homological and the GSV indices for this collection of 1-forms agree:
$$
\indhom(\{\omega_j^{(i)}\};X,0)=\indGSV(\{\omega_j^{(i)}\};X,0).
$$
\end{corollary}

%%My remark (S.G.). For $(X,0)=(\C^n,0)$ this is already proved at the end of section 4.
%%Thus one has to prove this for ICIS.

%% \begin{proof}
%% Without loss of generality, we can assume $(X,0)=(\C^n,0)$, so $N=n$. By Theorem \ref{vanishing one set},  the complex $\CC^{(i)}$ does not have higher homology. Its zeroth homology equals
%% $$
%% \Omega^n(\C^n,0)/\left(\omega_1^{(i)}\wedge \ldots\wedge \omega_{n-k_i+1}^{(i)}\wedge \Omega^{k_i-1}(\C^n,0)\right)\simeq \CO_{\C^n,0}/I^{(i)},
%% $$
%% where the ideal $I^{(i)}$ is generated by the $(n-k_i+1)\times (n-k_i+1)$-minors of the $n\times n-k_i+1$ matrix of coefficients 
%% of the forms $\omega_j^{i}$. Furthermore, by Theorem \ref{th:vanishing} the only nontrivial homology of the complex $\CC$
%% equals
%% $$
%% H^0(\CC)=\CO_{\C^n,0}/\left(I^{(1)}+\ldots+I^{(s)}\right).
%% $$ 
%% On the other hand, by  definition \eqref{def gsv collection} the GSV index of $\{\omega_j^{(i)}\}$ equals the same thing.
%% \end{proof}
\begin{proof}
By Theorem \ref{th:vanishing}, the complex $\wCC^{(i)}$ does not have higher homologies.
Its zeroth homology equals
$$
\CO_{X,0}\otimes \wedge^N\C^N/\left(df_1\cdots df_{N-n}\wedge \omega^{(i)}_1\cdots\omega^{(i)}_{n-k_i+1}\wedge
\CO_{X,0}\otimes \wedge^{k_i-1}\C^N\right)
\simeq \CO_{X,0}/I^{(i)},
$$
where the ideal $I^{(i)}$ is generated by the $(N-k_i+1)\times (N-k_i+1)$-minors of the
$N\times (N-k_i+1)$ matrix of coefficients 
of the forms $df_1$, \dots $df_{N-n}$, $\omega^{(i)}_1$, \dots, $\omega^{(i)}_{n-k_i+1}$.
%%Furthermore, by Theorem \ref{th:vanishing} the only nontrivial homology of the complex $\wCC$
%%equals
%%$$
%%H^0(\wCC)=\CO_{X,0}/\left(I^{(1)}+\ldots+I^{(s)}\right).
%%$$ 
Therefore the zeroth homology of the complex $\wCC\left(\{\omega_j^{(i)}\}\right)$ is equal to
$\CO_{X,0}/\sum_i I^{(i)}$.
On the other hand, by Equation \eqref{def gsv collection} the GSV-index of the collection
$\{\omega_j^{(i)}\}$ equals the dimension of the same space, so
$$
\indGSV(\{\omega_j^{(i)}\};X,0)=\chi\left(\wCC\left(\{\omega_j^{(i)}\}\right)\right).
$$
Therefore, it is sufficient to prove that the Euler characteristics of the complexes $\wCC$ and $\CC$ agree.

Indeed, let $\wCC'$ denote the complex \eqref{bicomplex} where $R\otimes \wedge^N\C^N$ is replaced by $\Imm(\phi)$.
Then by Lemma \ref{lem tau} we have:
$$
\CC^{(i)}/\T'\simeq \wCC'^{(i)},\ \wCC^{(i)}/\wCC'^{(i)}\simeq \T,
$$
so 
$$
\chi(\CC)=\sum_{A\subset \{1,\ldots,s\}}\chi\left(\T'^{\otimes |A|}\otimes \bigotimes_{i\notin A}\wCC'^{(i)}\right)=
$$
$$
\sum_{A\subset \{1,\ldots,s\}}\chi\left(\T^{\otimes |A|}\otimes \bigotimes_{i\notin A}\wCC'^{(i)}\right)=\chi(\wCC).
$$
Note that since $\T$ and $\T'$ are supported at the origin, the cohomology of all complexes in the sum are also supported at the origin, and
hence are finite-dimensional. Therefore all the Euler characteristics are finite.
\end{proof}

%%%%%%%%%%%%%%%%%%%%%%%%%%%%%%%%%%%%%%%%%%%%%%%%%
\section{Invariants of isolated singularities.}
%%%%%%%%%%%%%%%%%%%%%%%%%%%%%%%%%%%%%%%%%%%%%%%%%
Let $(X,0)\subset (\C^N, 0)$ be a germ of a complex analytic variety of pure dimension $n$
with an isolated singular point at the origin, let $k_i$, $i=1,\ldots, s$, be positive
integers such that $\sum_i k_i=n$, and let $\{\omega_j^{(i)}\}$, $i=1,\ldots, s$, $j=1,\ldots,n-k_i+1$,
be a collection of holomorphic 1-forms on $(X,0)$ without singular points on a punctured neighbourhood of the origin
in $X$. Since both the homological index and the Chern obstruction satisfy the law of conservation of number
and coincide with each other (and with the usual index) on a smooth manifold,
one has the following statement.

\begin{proposition}
 The difference $\indhom(\{\omega_j^{(i)}\}; X,0)-\Ch(\{\omega_j^{(i)}\}; X,0)$ between
 the homological index and the Chern obstruction of a collection of 1-forms on an isolated
 $n$-dimensional singularity does not depend on the collection of 1-forms $\{\omega_j^{(i)}\}$.
\end{proposition}

Thus this difference is an invariant of the singularity $(X,0)$.
For a generic collection $\{\omega_j^{(i)}\}$ (say, for the collection of differentials of
a generic collection of linear functions on $\C^N$), the Chern obstruction $\Ch(\{\omega_j^{(i)}\}; X,0)$
is equal to zero: see~\cite[Proposition 1.1]{EG-collections}. Therefore this difference is equal to the
homological index of a generic collection of 1-forms on $(X,0)$.
If $(X,0)$ is an ICIS and $s=1$ (and therefore $k_1=k$), this index is equal to $\mu(X)+\mu'(X)$, where $\mu(X)$ is the Milnor number
of the ICIS $(X,0)$ (the rank of the middle homology group of its smoothing) and $\mu'(X)$ is
the Milnor number of its generic hyperplane section. If $(X,0)$ is an isolated hypersurface singularity,
up to a constant factor (half of the volume of the $2n$-dimensional sphere) it is equal to the limit
of the integral over the Milnor fibre of the Gauss curvature of the fibre: \cite{Langevin}.
Thus one can say that (up to
a constant factor) it is equal to ``the vanishing curvature of the singularity''.
One can conjecture that the introduced invariants in a similar way are related with vanishing
integrals of the forms defined in terms of the curvature tensor and giving the corresponding
Chern characteristic numbers.


\begin{thebibliography}{15}

%% \bibitem{BE} D.~Buchsbaum, D.~Eisenbud. Remarks on ideals and resolutions. Symposia Mathematica,
%% Vol. XI (Convegno di Algebra Commutativa, INDAM, Rome, 1971), pp. 193--204. Academic Press, London, 1973.

\bibitem{EG-BLMS-2005} W.~Ebeling, S.M.~Gusein-Zade. Indices of vector fields or 1-forms and characteristic
numbers. Bull. London Math. Soc. {\bf 37} (2005), no.5, 747--754.

\bibitem{EG-Survey} W.~Ebeling, S.M.~Gusein-Zade. Indices of vector fields and 1-forms on singular varieties.
In: Global aspects of complex geometry, 129--169, Springer, Berlin, 2006.

\bibitem{EG-collections} W.~Ebeling, S.M.~Gusein-Zade. Indices of collections of 1-forms on singular varieties.
In: Singularities in Geometry and Topology,
%% : Proceedings of the Trieste Singularity Summer School and Workshop, ICTP, Trieste,
%% Italy, 15 August - 3 September 2005 (J.-P.Brasselet, J.Damon, Le Dung Trang, M.Oka, Editors),
World Scientific, 2007, 629--639.

\bibitem{EG-MMJ} W.~Ebeling, S.M.~Gusein-Zade. Indices of 1-forms on an isolated complete intersection
singularity. Mosc. Math. J. {\bf 3} (2003), no.2, 439--455.

\bibitem{EGS} W.~Ebeling, S.M.~Gusein-Zade, J.~Seade. Homological index for 1-forms and
a Milnor number for isolated singularities. Internat. J. Math. {\bf 15} (2004), no.9, 895--905. 

\bibitem{EN} J.A.~Eagon, D.G.~Northcott. Ideals defined by matrices and a certain complex associated with them. Proc. Roy. Soc. Ser. A {\bf 269} (1962) 188--204.

\bibitem{Eis} D.~Eisenbud. Commutative algebra. With a view toward algebraic geometry. Graduate Texts in Mathematics, {\bf 150}. Springer-Verlag, New York, 1995.

%% \bibitem{G-M} X.~G\'omez-Mont. An algebraic formula for the index of a vector field on a hypersurface
%% with an isolated singularity. J. Algebraic Geom. {\bf 7} (1998), no.4, 731--752. 

\bibitem{GG} L.~Giraldo, X.~G\'omez-Mont. A law of conservation of number for local Euler characteristics.
Complex manifolds and hyperbolic geometry (Guanajuato, 2001), 251--259, Contemp. Math., 311, Amer. Math.
Soc., Providence, RI, 2002.

\bibitem{Greuel} G.-M.~Greuel. Der Gau\ss-Manin-Zusammenhang isolierter Singularit\"aten von
vollst\"andigen Durchschnitten. Math. Ann. {\bf 214} (1975), 235--266.

\bibitem{GM} S.M.~Gusein-Zade, F.I. Mamedova.
On equivariant indices of 1-forms on varieties.
Funct. Anal. Appl. {\bf 51} (2017), no.2; %%, ??-??.
arXiv: 1701.01827.

\bibitem{Langevin} R.~Langevin. Courbure et singularit\'es complexes. Comment. Math. Helv. {\bf 54} (1979),
no.1, 6--16. 

\bibitem{Weyman} J. Weyman. Resolutions of the Exterior and Symmetric Powers of a Module. Journal of Algebra {\bf 58} (1979), 333-341.

\end{thebibliography}
\end{document}